\documentclass{amsart}
\usepackage{amsmath, amsthm, amssymb}
\usepackage{verbatim}
\usepackage[retainorgcmds]{IEEEtrantools}
\usepackage{hyperref}

\usepackage{enumerate}
\usepackage[OT4]{polski}
\usepackage[T1]{fontenc} 
\usepackage[utf8]{inputenc}
\usepackage[english]{babel} 

\usepackage{tikz}
\usetikzlibrary{calc}

\usepackage{kantlipsum}
\allowdisplaybreaks

\theoremstyle{plain}
    \newtheorem{theorem}{Theorem}[section]

    \newtheorem{lemma}[theorem]{Lemma}
    \newtheorem{corollary}[theorem]{Corollary}
    
\theoremstyle{definition}
    \newtheorem{definition}[theorem]{Definition}
    
\theoremstyle{remark}

\title{On the Minimum Area of Null Homotopies of Curves Traced Twice}
\author{Zipei Nie}
\address{Z. Nie: Department of Mathematics, Massachusetts Institute of Technology, Cambridge MA}
\email{zipei@mit.edu} 
\date{\today}
\begin{document}
\begin{abstract}We provide an efficient algorithm to compute the minimum area of a homotopy between two closed plane curves, given that they divide the plane into finite number of regions. For any positive real number $\varepsilon>0$, we construct a closed plane curve $\gamma$ such that the minimum area of a null homotopy of $2\cdot\gamma$ is less than $\varepsilon$ times that of $\gamma$. We also establish a lower bound on how complex a desired closed curve has to be.
\end{abstract}
\maketitle

\section{Introduction}
The problems related to the minimum area of a null homotopy of a closed curve traced multiple times have been studied for decades along with the Plateau's problem. In 1968, L. C. Young \cite{lyoung} constructed a closed curve $\gamma$ in the Euclidean space $\mathbb{R}^4$ such that the least area of a surface with boundary $2\cdot\gamma$ is strictly less than twice of that with boundary $\gamma$ in the context of integral current. Roughly speaking, it is not necessarily twice as hard to fill a curve traced twice as to fill the curve traced only once with a surface in $\mathbb{R}^4$. His example was inspired by the embedding of the Klein bottle into $\mathbb{R}^4$. In 1984, F. Morgan \cite{fmorgan} and B. White \cite{bwhite} generalized his example to the $n\cdot \gamma$ case for any integer $n\ge 2$ via different approaches from each other. 

Recently in 2013, R. Young \cite{ryoung} proved that for a $(d-1)$-cycle $T$ in $\mathbb{R}^n$, there exists a positive constant $c$ depending on $d$ and $n$ such that $\mathrm{FV}(T)\le c \mathrm{FV}(2\cdot T)$ where the filling volume $\mathrm{FV}(T)$ is defined as the infimum of the mass of a $d$-cycle $U$ such that $\partial U=T$. It is an inequality in the opposite direction to the previous constructions.

This paper concerns a homotopic analog of R. Young's work in the case where the dimension $d=2$. This case particularly interests us because the homotopy group $\pi_{n}(X)$ is commutative when $n>1$, but not necessarily so when $n=1$. In contrast to his result, we construct a closed curve $\gamma$ in $\mathbb{R}^2$, such that the minimum area of a null homotopy of $2\cdot\gamma$ is less than $\varepsilon$ times that of $\gamma$ for any positive constant $\varepsilon$ (see Theorem \ref{main-curve}). This construction was inspired by a highly twisted simple closed curve on the real projective plane we found during the research project.

In 2013, E. W. Chambers and Y. Wang \cite{sim} investigated the computation of the minimum area of a regular homotopy between two simple curves on $2$-manifolds with the same endpoints. In our paper, a different method allows us to compute the minimum area of a homotopy, not necessarily regular, between two closed plane curves, not necessarily simple, in polynomial time and space, given that they divide the plane into finite number of regions (see Theorem \ref{distance-complexity} and Theorem \ref{main-relation}). It is worth noting that the computation of the minimum area of a homotopy between two curves $\gamma_1$ and $\gamma_2$ with the same endpoints can be converted to the computation of the minimum area of a null homotopy of the closed curve $\gamma_1\cdot \overline{\gamma_2}$. The main obstacle when trying to generalize our computation to the case on $2$-manifolds is that we currently do not have an efficient algorithm to compute the cancellation norm under more general settings, as commented by M. Brandenbursky, \'S. R. Gal, J. K\k{e}dra and M. Marcinkowski \cite{norm}.

This paper is organized as follows.

In Section 2, we define the weighted cancellation norm and prove several properties of this pseudo-norm, which is a natural generalization of the cancellation norm defined by S. Gadgil \cite{rna} and M. Brandenbursky, \'S. R. Gal, J. K\k{e}dra and M. Marcinkowski \cite{norm}.

In Section 3, we define the weighted cancellation distance and compare it to the weighted cancellation distance.

In Section 4, we establish a relation between the weighted cancellation distance and the minimum area of homotopy between two closed plane curves.

In Section 5, we construct a closed plane curve that satisfies our conditions.

In Section 6, we provide a lower bound on how complex a desired closed curve has to be.

\section{Weighted Cancellation Norm}

In recent papers \cite{norm,rna}, the authors introduced an interesting conjugacy invariant norm on the free group. This norm could be generalized as follows without much efforts. For our convenience, we change the ordered pairs in the definition of a folding of a word in \cite{rna} into unordered pairs.

\begin{definition}
A symmetric set without identity $S$ is a set with a bijective function $\mathrm{inv}: S\rightarrow S$ with no fixed points such that $\mathrm{inv}(\mathrm{inv}(x))=x$ for any $x\in S$.
\end{definition}

Note that this definition is compatible with the definition of the symmetric set as a subset of a group, where we can take the function $\mathrm{inv}$ to be the inverse of a group element. By the axiom of choice, there exists a subset $T$ of $S$ such that $S$ is the union of the disjoint sets $T$ and $\mathrm{inv}(T)$. Conversely, any set $T$ induces a symmetric set without identity $S=T\sqcup T$ naturally. If we consider $S$ as a subset of the group $F_T$ freely generated by $T$, then it is a symmetric set as a subset of the group. 

\begin{definition}
We say that two disjoint pairs $\{i ,j\}$ and $\{k,m\}$ where $i<j$ and $k<m$, are linked if either $i<k<j<m$ or $k<i<m<j$.
\end{definition}

\begin{definition}
A folding of a word $w=l_1\cdots l_n$ over a symmetric set without identity is a collection of disjoint pairs $\mathcal{F}\subset \{ \{i,j\}|1\le i<j\le n\}$ such that for $\{i,j\}\in \mathcal{F}$ we have $l_i=\mathrm{inv}(l_j)$ and that any two pairs $\{i_1,j_1\}\in \mathcal{F}$ and $\{i_2,j_2\} \in \mathcal{F}$ are not linked.  
\end{definition}

\begin{definition}
We say a position $i$ ($1\le i\le n$) is unpaired in a folding $\mathcal{F}$ of a word $w=l_1\cdots l_n$ over a symmetric set without identity, if $\{i,j\}\not\in\mathcal{F}$ for any $1\le j\le n$.
   
\end{definition}

\begin{definition}
A weight function on a symmetric set without identity $S$ is a function $\mathrm{wt}:S\rightarrow \mathbb{R}^+\cup\{0\}$ such that $\mathrm{wt}(\mathrm{inv}(x))=\mathrm{wt}(x)$ for any $x\in S$. 
\end{definition}

In case the symmetric set without identity $S$ is the disjoint union of subsets $T$ and $\mathrm{inv}(T)$, a weight function on $S$ corresponds to a function from $T$ to $\mathbb{R}^+\cup\{0\}$ by restriction.

\begin{definition}The weighted cancellation norm of a word $w=l_1\cdots l_n$ over a symmetric set without identity $S$ with a weight function $\mathrm{wt}$ is defined as 
$$||w||:=\min_{\mathcal{F}}\sum_{i} \mathrm{wt}(l_i)$$
where $\mathcal{F}$ ranges over all foldings of $w$ and $i$ ranges over all unpaired positions in $\mathcal{F}$. 
\end{definition}

As a generalization of Lemma 2.J in \cite{norm}, the following theorem shows that there exists an efficient algorithm to compute the weighted cancellation norm of a word. 

\begin{theorem}\label{norm-complexity}
The weighted cancellation norm of a word $w=l_1\cdots l_n$ can be computed within time $O(n^3)$ and space $O(n^2)$.
\end{theorem}
\begin{proof}
For $n\ge 2$, consider a folding $\mathcal{F}$ of the word $w$. If $n$ is an unpaired position in the folding $\mathcal{F}$, then $\mathcal{F}$ is also a folding of the word $w'=l_1\cdots l_{n-1}$. Otherwise assume there exists an integer $k$ ($1\le k \le n-1$) such that $\{k,n\}\in \mathcal{F}$, then the set $\mathcal{F}_1=\{\{i,j\}| \{i,j\}\in \mathcal{F}\mbox{ and } i,j<k\}$ is a folding of the word $w_1=l_1\cdots l_{k-1}$ and the set $\mathcal{F}_2=\{\{i-k,j-k\}| \{i,j\}\in \mathcal{F}\mbox{ and } i,j>k\}$ is a folding of the word $w_2=l_{k+1}\cdots l_{n-1}$. Conversely, if we have a folding $\mathcal{F'}$ of the word $w'=l_1\cdots l_{n-1}$, then $\mathcal{F'}$ is also a folding of the word $w=l_1\cdots l_n$. And if we a folding $\mathcal{F}_1$ of the word $w_1=l_1\cdots l_{k-1}$ and a folding $\mathcal{F}_2$ of the word $w_2=l_{k+1}\cdots l_{n-1}$ and $l_k=\mathrm{inv}(l_n)$, then the set $\mathcal{F}=\mathcal{F}_1\cup\{\{i+k,j+k\}|\{i,j\}\in \mathcal{F}_2\} \cup \{\{k,n\}\}$ is a folding of the word $w$. Therefore $$||w||=\min \{||w'||+\mathrm{wt}(l_n),\min_k\{||w_1||+||w_2||\}\}$$ where $w'=l_1\cdots l_{n-1}$, $w_1=l_1\cdots l_{k-1}$, $w_2=l_{k+1} \cdots l_{n-1}$ and $k$ ranges over all integers between $1$ and $n-1$ with $l_k=\mathrm{inv}(l_n)$. 

Let $m$ be an integer with $1\le m\le n$. If we know the weighted cancellation norms of the words of the form $l_i\cdots l_j$ where $1\le i\le j\le m-1$, then the weighted cancellation norms of the words of the form $l_i\cdots l_m$ where $1\le i\le m$ can be computed within time $O(m^2)$. Therefore the weighted cancellation norm of $w$ can be computed within time $O(n^3)$ and space $O(n^2)$.
\end{proof}

The following theorem checks the well-definedness of Definition \ref{norm-definition}.

\begin{theorem}\label{norm-defining-property} Suppose $S$ is a symmetric set without identity with a weight function $\mathrm{wt}$, and $S$ is the union of disjoint subsets $T$ and $\mathrm{inv}(T)$. Let $F_T$ be the group freely generated by $T$. Then for an element $g\in F_T$, the formula $$||g||:=||w||$$
where $w$ is a word over the alphabet $S$ representing $g$ gives a well-defined conjugacy invariant pseudo-norm on $F_T$.
\end{theorem}

\begin{proof}
First, we prove if two words represents the same element in the group $F_T$, then they have the same weighted cancellation norms. Since the group $F_T$ is free, it suffices to prove the words $w=l_1\cdots l_n$ and $w'=l_1\cdots l_t \cdot l\cdot  \mathrm{inv}(l)\cdot l_{t+1}\cdots l_n$ over the alphabet $S$ have the same weighted cancellation norms. Let $\phi:\{1,\cdots,n\}\rightarrow\{1,\cdots,n+2\}$ be an auxiliary function defined by $$\phi(i)=\begin{cases}i&\mbox{, if }1\le i\le t \\i+2&\mbox{, if }t+1\le i\le n .\end{cases}$$ On the one hand, for any folding $\mathcal{F}$ of the word $w$, the set $\mathcal{F'}=\{\{\phi(i),\phi(j)\}|\{i,j\}\in\mathcal{F}\}\cup\{\{t+1,t+2\}\}$ is a folding of the word $w'$. Then the sums of weights of the letters at unpaired positions are the same for the folding $\mathcal{F}$ of the word $w$ and the folding $\mathcal{F'}$ of the word $w'$, hence $||w||\ge||w'||$. On the other hand, let $\mathcal{F'}$ be a folding of the word $w'$. If $t+1$ or $t+2$ is an unpaired position in the folding $\mathcal{F'}$ or $\{t+1,t+2\}\in \mathcal{F'}$, then the set $\mathcal{F}=\{\{i,j\}|\{\phi(i),\phi(j)\}\in\mathcal{F'}\}$ is a folding of the word $w$ with the sum of weights of the letters at unpaired positions no more than that for the folding $\mathcal{F'}$. Otherwise assume that $\{t+1,j_1\}\in \mathcal{F'}$ and $\{t+2,j_2\}\in\mathcal{F'}$, then the set $\mathcal{F}=\{\{i,j\}|\{\phi(i),\phi(j) \}\in\mathcal{F'}\}\cup\{\{j_1,j_2\}\}$ is a folding of the word $w$ with the sum of weights of the letters at unpaired positions the same as that for the folding $\mathcal{F'}$. Therefore $||w||\le ||w'||$. 

Second, we prove the weighted cancellation norm of the concatenation $w_1w_2$ of two words $w_1=l_1\cdots l_n$ and $w_2=l'_1\cdots l'_m$ over the alphabet $S$ is at most $||w_1||+||w_2||$. Let $\mathcal{F}_1$ (resp., $\mathcal{F}_2$) be a folding of the word $w_1$ (resp., $w_2$). Then the set $\mathcal{F}=\mathcal{F}_1\cup\{\{n+i,n+j\}|\{i,j\}\in \mathcal{F}_2\}$ is a folding of $w$ with the sum of weights of the letters at unpaired positions equal to the sum of those for the foldings $\mathcal{F}_1$ and $\mathcal{F}_2$. Therefore $||w_1 w_2||\le ||w_1||+||w_2||$.

Then, we prove the weighted cancellation norms of the word $w=l_1\cdots l_n$ and the word $w'=\mathrm{inv}(l_n)\cdots \mathrm{inv}(l_1)$ over the alphabet $S$ are the same. Let $\mathcal{F}$ be a folding of $w$, then $\mathcal{F'}=\{\{n-i+1,n-j+1\}|\{i,j\}\in\mathcal{F}\}$ is a folding of $w'$ with the same sum of weights of the letters at unpaired positions, so $||w||\ge ||w'||$. Similarly we have $||w||\le||w'||$.

Last, we prove for two words $w_1=l_1\cdots l_n$ and $w_2=l'_1\cdots l'_m$ over the alphabet $S$, we have $||w_1w_2||=||w_2w_1||$. Let $\phi:\{1,\cdots, m+n\}\rightarrow\{1,\cdots,m+n\}$ be auxiliary function defined by $$\phi(i)=\begin{cases} i+m &\mbox{, if }1\le i\le n \\ i-n &\mbox{, if } n+1\le i \le m+n.\end{cases}$$ For any folding $\mathcal{F}$ of the word $w_1 w_2$, the set $\mathcal{F'}=\{\{\phi(i),\phi(j)\}|\{i,j\}\in \mathcal{F}\}$ is a folding of the word $w_2w_1$. So $||w_1w_2||\ge ||w_2w_1||$. Similarly we have $||w_1w_2||\le ||w_2w_1||$.
\end{proof}

Now we can speak about the weighted cancellation norm on the free group. If we have $\mathrm{wt}(x)=1$ for any $x\in S$, then the weighted cancellation norm degeneralizes to the cancellation norm on free groups defined as in \cite{norm}.

\begin{definition}\label{norm-definition}
Given a set $T$ and a function $\mathrm{wt}: T\rightarrow \mathbb{R}^+\cup\{0\}$. The weighted cancellation norm on the group $F_T$ freely generated by $T$ is defined as $$||g||:=||w||$$
where $g$ is any element in $F_T$ represented by a word $w$ over the alphabet $S$, and $S=T\sqcup T$ is the symmetric set without identity induced by $T$ with a weight function induced by $\mathrm{wt}$.
\end{definition}

Our next theorem provides us with an equivalent description of the weighted cancellation norms on free groups.  

\begin{theorem}\label{norm-equi-form}
Given a set $T$ and a function $\mathrm{wt}: T\rightarrow \mathbb{R}^+\cup \{0\}$. Let $F_T$ be the group freely generated by $T$ and $||\cdot||$ be the weighted cancellation norm on it. Then for $g\in F_T$, $||g||$ is the smallest real number $x$ such that there exist a nonnegative integer $m$, elements $v_1,\cdots, v_m\in F_T$ and $t_1, \cdots, t_m\in T\subseteq F_T$ such that $g=v_1\cdots v_m$ where $v_j$ is a conjugate of $t_j$ or $t_j^{-1}$ for each $j=1,\cdots, m$ and $\sum_{j=1}^{m} \mathrm{wt}(t_j)=x$.
\end{theorem}
\begin{proof}
On one hand, assume we have such a decomposition of $g$ as in the statement. Then as a word over the symmetric set without identity induced by $T$, the weighted cancellation norm of $t_j$ or $\mathrm{inv}(t_j)$ is equal to $\mathrm{wt}(t_j)$ for $j=1,\cdots, m$. So by Theorem \ref{norm-defining-property}, we have $x=\sum_{j=1}^{m} \mathrm{wt}(t_j)=\sum_{j=1}^m||v_j||\ge ||v_1\cdots v_m||=||g||$.

On the other hand, by definition there exists a folding $\mathcal{F}$ of the word representing $g$ with the sum of weights of the letters at unpaired positions equal to $||g||$. So $g$ can be written as the product $h_0 s_1 h_1 s_2\cdots h_m$ of the elements $h_0, h_1,\cdots, h_m, s_1,\cdots, s_m$ in $F_T$ where $m$ is a nonnegative integer and $s_j$ is equal to $t_j$ or $t_j^{-1}$ for some $t_j\in T\subseteq F_T$ for each $j=1,\cdots, m$ such that $h_0 h_1\cdots h_m=1$ and $\sum_{j=1}^m \mathrm{wt}(t_j)=||g||$. Let $v_j=(h_0\cdots h_{j-1})s_j(h_0\cdots h_{j-1})^{-1}$ for $j=1\cdots m$, then $v_1\cdots v_m= h_0s_1h_1 s_2\cdots s_{m} (h_0\cdots h_{j-1})^{-1}=g(h_0\cdots h_m)^{-1}=g$.
\end{proof}

\section{Weighted Cancellation Distance}

This section is devoted to the concept of the weighted cancellation distance, which is a useful analog of the weighted cancellation norm. 

\begin{definition}
Let $w_1=l_1\cdots l_n$ and $w_2=l'_1\cdots l'_m$ be two words over a symmetric set without identity. A mixed folding of the words $w_1$ and $w_2$ is a tuple $(\mathcal{F},p,q)$ where $p$ ($0\le p \le n-1$) and $q$ ($0\le q \le m-1$) are integers and $\mathcal{F}$ is a folding of the word $w=l_{p+1}\cdots l_n \cdot l_1\cdots l_p\cdot \mathrm{inv}(l'_{q})\cdots \mathrm{inv}(l'_1)\cdot \mathrm{inv}(l'_m)\cdots \mathrm{inv}(l'_{q+1})$.  
\end{definition}

\begin{definition}
Let $w_1=l_1\cdots l_n$ and $w_2=l'_1\cdots l'_m$ be two words over a symmetric set without identity. An unpaired position in a mixed folding $(\mathcal{F},p,q)$ of the words $w_1$ and $w_2$ in the first word is an integer $i$ ($1\le i \le n$) such that $\phi_1(i)$ is an unpaired position in the folding $\mathcal{F}$ of the word $w=l_{p+1}\cdots l_n \cdot l_1\cdots l_p\cdot \mathrm{inv}(l'_{q})\cdots \mathrm{inv}(l'_1)\cdot \mathrm{inv}(l'_m)\cdots \mathrm{inv}(l'_{q+1})$ where $$\phi_1(i)=\begin{cases}n-p+i &\mbox{ , if }1\le i\le p\\i-p &\mbox{ , if } p+1\le i\le n \end{cases}$$ is an auxiliary function. An unpaired position in the second word is an integer $j$ ($1\le j \le m$) such that $\phi_2(j)$ is an unpaired position in the folding $\mathcal{F}$ of the word $w$ where $$\phi_2(j)=\begin{cases}n+q-j+1 &\mbox{ , if }1\le j\le q\\m+n+q-j+1&\mbox{ , if } q+1\le j\le m \end{cases}$$ is an auxiliary function.
\end{definition}

\begin{definition}
The weighted cancellation distance between two words $w_1=l_1\cdots l_n$ and $w_2=l'_1\cdots l'_m$ over a symmetric set without identity $S$ with a weight function $\mathrm{wt}$ is defined as $$d(w_1,w_2):=\min_{(\mathcal{F},p,q)}\{\sum_i \mathrm{wt}(l_i)+\sum_j \mathrm{wt}(l'_j)\}$$ where $(\mathcal{F},p,q)$ ranges over all mixed foldings of $w_1$ and $w_2$ and $i$ (resp., $j$) ranges over all unpaired positions in the mixed folding $(\mathcal{F},p,q)$ in the first (resp., second) word. 
\end{definition}

\begin{theorem}\label{distance-complexity}
The weighted cancellation distance between two words $w_1=l_1\cdots l_n$ and $w_2=l'_1\cdots l'_m$ can be computed within time $O(N^5)$ and space $O(N^2)$, where $N=\max\{n,m\}$.
\end{theorem}
\begin{proof}
By definition, we know 
\begin{equation*} 
\begin{aligned} 
d(w_1,w_2)&=\min_{(\mathcal{F},p,q)}\{\sum_i \mathrm{wt}(l_i)+\sum_j \mathrm{wt}(l'_j)\}\\ 
&=\min_{p}\min_{q}\min_{\mathcal{F}}\{\sum_i \mathrm{wt}(l_i)+\sum_j \mathrm{wt}(l'_j)\}\\
&=\min_p\min_q ||w_{p,q}||
\end{aligned}
\end{equation*}
where $(\mathcal{F},p,q)$ ranges over all mixed foldings of $w_1$ and $w_2$, and $i$ (resp., $j$) ranges over the unpaired positions in $(\mathcal{F},p,q)$ in the first (resp., second) word, and $w_{p,q}=l_{p+1}\cdots l_n \cdot l_1\cdots l_p\cdot \mathrm{inv}(l'_{q})\cdots \mathrm{inv}(l'_1)\cdot \mathrm{inv}(l'_m)\cdots \mathrm{inv}(l'_{q+1})$ is a word of length $n+m\le 2N$. Since there are $mn\le N^2$ possible pairs $(p,q)$, by Theorem \ref{norm-complexity}, we can compute $d(w_1,w_2)$ within time $O(N^5)$ and space $O(N^2)$.
\end{proof}

The following theorem shows that the analog of Definition \ref{norm-definition} in the weighted cancellation distance case can be concisely stated as in Definition \ref{distance-definition}.

\begin{theorem}\label{distance-defining-property}
Suppose $S$ is a symmetric set without identity with a weight function $\mathrm{wt}$, and $S$ is the union of disjoint subsets $T$ and $\mathrm{inv}(T)$. Let $F_T$ be the group freely generated by $T$. Then for two elements $g_1,g_2\in F_T$, the formula $$d(g_1,g_2):=d(w_1, w_2)$$
where $w_1$ (resp., $w_2$) is a word over the alphabet $S$ representing $g_1$ (resp., $g_2$) gives a well-defined pseudo-metric on the group $F_T$, and the following equality $$d(g_1,g_2)=\min_{u\in F_T}||ug_1u^{-1}g_2^{-1}||$$ holds for any $g_1,g_2\in F_T$.
\end{theorem}

\begin{proof}

First, we prove for two words $w_1=l_1\cdots l_n$ and $w_2=l'_1\cdots l'_m$, we have $d(w_1,w_2)=d(w_2,w_1)$. Assume that $(\mathcal{F},p,q)$ is a mixed folding of the words $w_1$ and $w_2$, so $\mathcal{F}$ is a folding of the word $w=l_{p+1}\cdots l_n \cdot l_1\cdots l_p\cdot \mathrm{inv}(l'_{q})\cdots \mathrm{inv}(l'_1)\cdot \mathrm{inv}(l'_m)\cdots \mathrm{inv}(l'_{q+1})$, then the tuple $(\mathcal{F'},q,p)$ is a mixed folding of $w_2$ and $w_1$ with the same sum of the weights of the letters at unpaired positions, where the set $\mathcal{F'}=\{\{m+n-i+1, m+n-j+1\}|\{i,j\}\in \mathcal{F}\}$. Therefore $d(w_1,w_2)\ge d(w_2,w_1)$. Similarly $d(w_1,w_2)\le d(w_2,w_1)$. In case $d(g_1,g_2)$ is well-defined for $g_1, g_2\in F_T$, it implies $d(g_1,g_2)=d(g_2,g_1)$ for any $g_1, g_2\in F_T$.

Second, we prove the formula for $d(g_1,g_2)$ does not depend on the choices of $w_1$ (resp., $w_2$) representing $g_1$ (resp., $g_2$). Since the group $F_T$ is free and the distance $d$ is symmetric over words, it suffices to prove that for the words $w_1=l_1\cdots l_n$, $w_2=l'_1\cdots l'_m$ and $w'_1=l_1\cdots l_t\cdot l \cdot \mathrm{inv}(l) \cdot l_{t+1}\cdots l_n$ over the alphabet $S$, we have $d(w_1,w_2)=d(w'_1,w_2)$. Let $$\phi(p,i)=\begin{cases} i&\mbox{, if } 1\le i \le p \\i+2 &\mbox{, if } p+1\le i \le m+n \end{cases}$$ be a auxiliary function. On one hand, let $(\mathcal{F},p,q)$ be a mixed folding of the words $w_1$ and $w_2$, then $\mathcal{F}$ is a folding of the word $w=l_{p+1}\cdots l_n \cdot l_1\cdots l_p\cdot \mathrm{inv}(l'_{q})\cdots \mathrm{inv}(l'_1)\cdot \mathrm{inv}(l'_m)\cdots \mathrm{inv}(l'_{q+1})$. If $0\le t\le p$, then the tuple $(\mathcal{F}_1,p+2,q)$ is a mixed folding of the words $w'_1$ and $w_2$ with the same sum of the weights of the letters at unpaired positions, where the set $\mathcal{F}_1=\{\{\phi(n-p+t,i),\phi(n-p+t,j)\}|\{i,j\}\in \mathcal{F}\}\cup\{\{n-p+t+1,n-p+t+2\}\}$. Similarly, if $p+1\le t\le n$, then the tuple $(\mathcal{F}_2,p,q)$ is a mixed folding of the words $w'_1$ and $w_2$ with the same sum of the weights of the letters at unpaired positions, where the set $\mathcal{F}_2=\{\{\phi(t-p,i),\phi(t-p,j)\}|\{i,j\}\in \mathcal{F}\}\cup\{\{t-p+1,t-p+2\}\}$. Hence we have $d(w_1,w_2)\ge d(w'_1,w_2)$. On the other hand, let $(\mathcal{F'},p,q)$ be a mixed folding of the words $w'_1$ and $w_2$. If $0\le t\le p-2$, then the tuple $(\mathcal{F}_1,p-2,q)$ is a mixed folding of the words $w_1$ and $w_2$ with the same sum of the weights of the letters at unpaired positions, where the set $\mathcal{F}'_1=\{\{i,j\}|\{\phi(n-p+t+2,i),\phi(n-p+t+2,j)\}\in\mathcal{F'}\}$. If $p\le t \le n$, then the tuple $(\mathcal{F}_2,p,q)$ is a mixed folding of the words $w_1$ and $w_2$ with the same sum of the weights of the letters at unpaired positions, where the set $\mathcal{F}'_2=\{\{i,j\}|\{\phi(t-p,i),\phi(t-p,j)\}\in\mathcal{F'}\}$. If $t=p-1$, then the tuple $(\mathcal{F}'_3, p-1,q)$ is a mixed folding of the words $w_1$ and $w_2$ with the same sum of the weights of the letters at unpaired positions, where the set $\mathcal{F}'_3=\{\{i,j\}|\{\phi_0(i),\phi_0(j)\}\in\mathcal{F'}\}$ and $$\phi_0(i)=\begin{cases} i+1 &\mbox{, if }1\le i \le n\\ i+2 &\mbox{, if }n+1\le i\le m+n\end{cases}$$ is an auxiliary function. Hence we have $d(w_1,w_2)\le d(w'_1,w_2)$.

Then, we prove the equality $d(g_1,g_2)=\min_{u\in F_T}||ug_1u^{-1}g_2^{-1}||$ holds for any $g_1,g_2\in F_T$. On one hand, let $(\mathcal{F},p,q)$ be a mixed folding of two words $w_1=l_1\cdots l_n$ and $w_2=l'_1\cdots l'_m$ representing $g_1$ and $g_2$, then $\mathcal{F}$ is a folding of the word $w=l_{p+1}\cdots l_n \cdot l_1\cdots l_p\cdot \mathrm{inv}(l'_{q})\cdots \mathrm{inv}(l'_1)\cdot \mathrm{inv}(l'_m)\cdots \mathrm{inv}(l'_{q+1})$. The word $l_{p+1}\cdots l_n \cdot l_1\cdots l_p$ represents a conjugate of $g_1$ in $F_T$ and the word $\mathrm{inv}(l'_{q})\cdots \mathrm{inv}(l'_1)\cdot \mathrm{inv}(l'_m)\cdots \mathrm{inv}(l'_{q+1})$ represents a conjugate of $g_2^{-1}$ in $F_T$. So there exists $u_1, u_2\in F_T$ such that the sum of the weights of the letters at unpaired positions in the mixed folding $(\mathcal{F},p,q)$ is no less than $||u_1g_1u_1^{-1}u_2 g_2^{-1} u_2^{-1}||=||u_2^{-1} u_1g_1u_1^{-1}u_2 g_2^{-1}||$. Hence there exists $u\in F_T$ such that $d(g_1,g_2)\ge ||ug_1u^{-1}g_2^{-1}||$. On the other hand, let $u$ be an element in $F_T$ represented by a word $l''_1\cdots l''_k$ over the alphabet $S$, and $\mathcal{F}$ be a folding of the word $l''_1\cdots l''_k\cdot l_1\cdots l_n\cdot \mathrm{inv}(l''_k)\cdots \mathrm{inv}(l''_1) \cdot \mathrm{inv}(l'_m)\cdots \mathrm{inv}(l'_1)$. Then the tuple $(\mathcal{F},n+k,0)$ is a mixed folding of the words $l_1\cdots l_n\cdot \mathrm{inv}(l''_k)\cdots \mathrm{inv}(l''_1)\cdot l''_1\cdots l''_k$ and $l'_1 \cdots l'_m$ representing $g_1$ and $g_2$ respectively. Hence for any $u\in F_T$, we have $d(g_1,g_2)\le||ug_1u^{-1}g_2^{-1}||$. Therefore $\min_{u\in F_T}||ug_1u^{-1}g_2^{-1}||$ exists and is equal to $d(g_1,g_2)$. As an immediate corollary, we know $d(g,g)=0$ for any $g\in F_T$.

Last, we prove for three group elements $g_1,g_2,g_3\in F_T$, we have $d(g_1,g_2)+d(g_2,g_3)\ge d(g_1,g_3)$. Let $u,v \in F_T$ be two group elements such that $d(g_1,g_2)=||ug_1 u^{-1} g_2^{-1}||$ and $d(g_2,g_3)=||vg_2 v^{-1} g_3^{-1}||$, then by Theorem \ref{norm-defining-property}, we have \begin{equation*} 
\begin{aligned}
d(g_1,g_2)+d(g_2,g_3) & = ||ug_1 u^{-1} g_2^{-1}||+||vg_2 v^{-1} g_3^{-1}|| \\
 &= ||ug_1 u^{-1} g_2^{-1}||+||g_2 v^{-1} g_3^{-1}v||\\
 &\ge||ug_1 u^{-1} v^{-1} g_3^{-1}v||\\
 &=||vug_1u^{-1}v^{-1} g_3^{-1}||\\
 &\ge d(g_1,g_3).
\end{aligned}
\end{equation*}

\end{proof}
Then we define the weighted cancellation distance between two elements on a free group. We remove the word ``weighted'' if the weight function $\mathrm{wt}$ satisfies $\mathrm{wt}(t)=1$ for each $t\in T$.  
\begin{definition}\label{distance-definition}
Given a set $T$ and a function $\mathrm{wt}: T\rightarrow \mathbb{R}^+\cup\{0\}$. The weighted cancellation distance $d: F_T\times F_T \rightarrow \mathbb{R}^+\cup \{0\}$ is defined as $$d(g_1,g_2)=\min_{u\in F_T}||ug_1u^{-1}g_2^{-1}||$$ where $g_1,g_2\in F_T$, and the group $F_T$ is freely generated by $T$.
\end{definition}

We use the properties of the weighted cancellation norms as a tool to prove some properties of the weighted cancellation distances.

\begin{theorem}\label{distance-property-II} Given a set $T$ and a function $\mathrm{wt}: T\rightarrow \mathbb{R}^+\cup\{0\}$. Let $F_T$ be the group freely generated by $T$. Then the following properties holds for the weighted cancellation distance $d$ where $g_1,g_2,g_3\in F_T$ and $n$ is a positive integer:
\begin{enumerate}[(a)]
\item $d(g_1,1)=||g_1||$;
\item $d(g_1,g_2)=d(g_1^{-1},g_2^{-1})$;
\item $d(g_1g_2,g_3)=d(g_2g_1,g_3)$;
\item $d(g_1^n,g_2^n)\le n\cdot  d(g_1,g_2)$.
\end{enumerate}
\end{theorem}

\begin{proof}
\begin{enumerate}[(a)]
\item This is immediate from the definition of the weighted cancellation distance and the conjugacy invariance of the weighted cancellation norm.
\item By definition 
\begin{equation*} 
\begin{aligned}
d(g_1^{-1},g_2^{-1})&=\min_{u\in F_T}||ug_1^{-1}u^{-1}g_2|| \\
&=\min_{u\in F_T}||u^{-1}g_2 ug_1^{-1}||\\
&=d(g_2,g_1)\\
&=d(g_1,g_2).
\end{aligned}
\end{equation*}
\item By definition \begin{equation*} 
\begin{aligned} 
d(g_1g_2,g_3)&=\min_{u\in F_T}||ug_1g_2u^{-1}g_3^{-1}|| \\
&=\min_{u\in F_T}||(ug_1)g_2g_1(ug_1)^{-1}g_3^{-1}||\\
&=d(g_2,g_1,g_3).
\end{aligned}
\end{equation*}
\item By definition there exists $u\in F_T$ such that $d(g_1,g_2)=||ug_1 u^{-1}g_2^{-1}||$. We use induction on $n$ to prove $||ug_1^n u^{-1}g_2^{-n}||\le n||u g_1 u^{-1} g_2^{-1}||$. The base case is trivial. Now assume $n\ge 2$ and $||ug_1^{n-1} u^{-1}g_2^{-n+1}||\le (n-1)||u g_1 u^{-1} g_2^{-1}||$. Then by Theorem \ref{distance-defining-property} we have
\begin{equation*} 
\begin{aligned} 
||ug_1^n u^{-1}g_2^{-n}||&=||g_2^{-1}ug_1^n u^{-1}g_2^{-n+1}||\\
&=||(g_2^{-1}ug_1 u^{-1})( u g_1^{n-1} u^{-1}g_2^{-n+1}) ||\\
&\le ||g_2^{-1}ug_1 u^{-1}||+|| u g_1^{n-1} u^{-1}g_2^{-n+1}||\\
&\le n||u g_1 u^{-1} g_2^{-1}||.
\end{aligned}
\end{equation*}
Because $d(g_1,g_2)=||ug_1 u^{-1}g_2^{-1}||$ and $d(g_1^n,g_2^n)\le||ug_1^n u^{-1}g_2^{-n}||$, we have $d(g_1^n,g_2^n)\le n\cdot  d(g_1,g_2)$.
\end{enumerate} 
\end{proof}

Lastly, we prove an analog of the Theorem \ref{norm-equi-form}.

\begin{theorem}\label{distance-equi-form} Given a set $T$ and a function $\mathrm{wt}:T\rightarrow \mathbb{R}^+\cup \{0\}$. Let $F_T$ be the group freely generated by $T$ and $d$ be the weighted cancellation distance on $F_T$. Then for $g_1, g_2\in F_T$, $d(g_1,g_2)$ is the smallest real number $x$ such that there exist a nonnegative integer $m$, elements $h_0, h_1,\cdots, h_m,u_0,u_1,\cdots,u_{m-1}\in F_T$ and $t_1,\cdots t_m\in T\subseteq F_T$ such that $h_0=g_1$, $h_m=g_2$, $u_j h_ju_j^{-1} h_{j+1}^{-1}$ is a conjugate of $t_{j+1}$ or $t_{j+1}^{-1}$ for each $j=0,\cdots, m-1$ and $\sum_{j=1}^{m} \mathrm{wt}(t_j)=x$.
\end{theorem}
\begin{proof}On one hand, if we have such a real number $x$, then by Theorem \ref{distance-defining-property}, we have $x=\sum_{j=0}^{m-1} \mathrm{wt}(t_j)=\sum_{j=0}^{m-1}||u_j h_j u_j^{-1} h_{j+1}^{-1}||\ge \sum_{j=0}^{m-1}d(h_j, h_{j+1})\ge d(h_0,h_m)=d(g_1,g_2)$.

On the other hand, by Theorem \ref{norm-equi-form}, for $g_1,g_2\in F_T$, there exists $u\in F_T$ such that $d(g_1,g_2)=||u g_1 u^{-1} g_2^{-1}||$. So there exist a nonnegative integer $m$, elements $v_1,\cdots,v_m\in F_T$ and $t_1,\cdots, t_m\in T\subseteq F_T$ such that $u g_1 u^{-1} g_2^{-1}=v_1\cdots v_m$ where $v_j$ is a conjugate of $t_j$ or $t_j^{-1}$ for each $j=1,\cdots, m$ and $\sum_{j=1}^{m} \mathrm{wt}(t_j)=d(g_1,g_2)$. Let $h_0=g_1$, $h_j=(v_1\cdots v_j)^{-1} ug_1 u^{-1}$ for each $j=1,\cdots, m$. Let $u_0=u$, $u_j=1$ for each $j=1,\cdots,m-1$. Then we have $u_0h_0 u_0^{-1} h_1^{-1}=ug_1u^{-1} (u g_1^{-1} u^{-1} v_1)=v_1$ and $u_j h_ju_j^{-1} h_{j+1}^{-1}= h_jh_{j+1}^{-1}=(v_1\cdots v_j)^{-1}(ug_1 u^{-1})( ug_1^{-1} u^{-1})(v_1\cdots v_{j+1})=v_{j+1}$ for $j=1,\cdots,m-1$.
\end{proof}

\section{Minimum Area of Homotopy}
In this section, we explain a close relation between the minimum area of a homotopy between two closed plane curves defined later, and the weight cancellation distance defined in Section 3.

Let $\gamma_1$ and $\gamma_2$ be two piecewise smooth closed curves on the plane, that is to say, injective piecewise smooth maps from $S^1=[0,1]/(0\sim 1)$ to $\mathbb{R}^2$. A homotopy between them is a piecewise smooth map $H:S^1\times[0,1]\rightarrow \mathbb{R}^2$ such that $H(s_1,0)=\gamma_1(s_1)$ and $H(s_2,1)=\gamma_2(s_2)$ for every $s_1,s_2\in S^1$. Similarly, let $\gamma_1$ and $\gamma_2$ be two piecewise smooth curves with the same endpoints, that is to say, injective piecewise smooth maps from $S^1=[0,1]$ to $\mathbb{R}^2$ with $\gamma_1(0)=\gamma_2(0)$ and $\gamma_1(1)=\gamma_2(1)$. Then a homotopy between them is a piecewise smooth map $H:[0,1]\times[0,1]\rightarrow \mathbb{R}^2$ such that $H(s_1,0)=\gamma_1(s_1)$, $H(s_2,1)=\gamma_2(s_2)$, $H(0,t_1)=\gamma_1(0)$ and $H(1,t_2)=\gamma_1(1)$ for every $s_1,s_2,t_1,t_2\in [0,1]$. In both cases, the area of the homotopy is defined as $$\mathrm{Area}(H):=\int\int|\det (DH(s,t))|ds dt.$$ 
If the homotopy $H$ is Lipschitz, then by the area formula in geometric measure theory, the above integral is finite and equal to $$\mathrm{Area}(H)=\int_{\mathbb{R}^2} |H^{-1}(\{x\})|dx.$$ 
By this formula for the area of homotopy, a homotopy between two piecewise smooth curves $\gamma_1$ and $\gamma_2$ with the same endpoints corresponds to a homotopy between the curve $\gamma_3$ and a constant curve with the same area, where $$\gamma_3(t)=\begin{cases}\gamma_1(2t)&\mbox{, if }0\le t\le \frac{1}{2}\\\gamma_2(2-2t)&\mbox{, if }\frac{1}{2}\le t \le 1\end{cases}$$ 
is a piecewise smooth curve on the plane.

Assume that two piecewise smooth closed curves $\gamma_1$ and $\gamma_2$ divide the plane into finite number of regions $R_0, R_1, \cdots R_n $ where $R_0$ is the unique infinite region. Let $F$ denote the fundamental group of the space $\mathbb{R}^2$ with every $R_i$ ($1\le i\le n $) removed. Let $A_i$ be the area of $R_i$ for each $1\le i\le n$. Then $F$ is freely generated by the elements $x_i$ ($1\le i \le n$), where each $x_i$ represents a generator of $\pi_1(\mathbb{R}^2-R_i)\cong \mathbb{Z}$. Let the weight function $\mathrm{wt}:\{x_1,\cdots,x_n\}\rightarrow \mathbb{R}^{+}$ be $\mathrm{wt}(x_i)=A_i$ for each $1\le i\le n$, then we have a weighted cancellation norm and a weighted cancellation distance on the group $F$. Let $g_1$ (resp., $g_2$) be the element in $F$ representing $\gamma_1$ (resp., $\gamma_2$). 

Under the above assumptions, we can prove the following theorem:

\begin{theorem}\label{main-relation}
The minimum area of a (piecewise smooth) homotopy between the closed curves $\gamma_1$ and $\gamma_2$ exists, and is equal to the weighted cancellation distance $d(g_1,g_2)$.
\end{theorem}

\begin{proof}
On one hand, let $H$ be a piecewise smooth homotopy between $\gamma_1$ and $\gamma_2$ on the plane. For each $1\le i\le n$, let $p_i\in R_i$ be a regular value of the homotopy $H$. Because there exist an nonnegative integer $m$ and group elements $h_j\in F$ ($0\le j\le m$) with $h_0=g_1$, $h_m=g_2$ and $d(h_j,h_{j+1})=A_i$ for $|H^{-1}(\{p_i\}) |$ times for each $i=1,\cdots, n$ where $j$ ranges over $0,\cdots,m-1$, we have $\sum_{i=1}^n A_i|H^{-1}(\{p_i\})| \ge d(g_1,g_2)$. By Sard's theorem, we have 
\begin{equation*}
\begin{aligned}
\mathrm{Area}(H)&=\int_{\mathbb{R}^2} |H^{-1}(\{x\})|dx \\
&\ge \sum_{i=1}^n\int_{R_i} |H^{-1}(\{x\}) |dx\\
&\ge \sum_{i=1}^n A_i \inf_{p_i} |H^{-1}(\{p_i\}) |\\
&\ge d(g_1,g_2)
\end{aligned}
\end{equation*}
where $p_i$ ranges over all regular values in $R_i$ of the homotopy $H$.

On the other hand, by Theorem \ref{distance-equi-form}, there exist an nonnegative integer $m$, group elements $h_j\in F$ ($0\le j\le m$) and $u_j\in F$ ($0\le j\le m-1$) such that $h_0=g_1$, $h_m= g_2$ and $u_jh_j u_j^{-1} h_{j+1}^{-1}$ ($0\le j\le m-1$) is a conjugate of $x_i$ or $x_i^{-1}$ where the integer $1\le i=i(j)\le n$ depends on $j$ and $\sum_{j=0}^{m-1} A_{i(j)}=d(g_1,g_2)$. We know the closed curve on $\mathbb{R}^2-R_i$ represented by $x_i$ or $x_i^{-1}$ is null homotopic on $\mathbb{R}^2$ by a homotopy with area $A_i$, so the closed curve on $\mathbb{R}^2-\cup_{i=0}^n R_i$ represented by $u_jh_j u_j^{-1} h_{j+1}^{-1}$ is null homotopic on $\mathbb{R}^2$ by a homotopy with area $A_i$. By changing the coordinates, the closed curves on $\mathbb{R}^2-\cup_{i=0}^n R_i$ represented by $h_j$ and $h_{j+1}$ are homotopic on $\mathbb{R}^2$ by a homotopy with area $A_i$. Therefore, there exists a homotopy with area $d(g_1,g_2)$ between the closed curves $\gamma_1$ and $\gamma_2$.
\end{proof}

The weighted cancellation distances can be computed in polynomial time and space by Theorem \ref{distance-complexity}, so the minimum area of a piecewise smooth homotopy between two piecewise smooth closed plane curves can be computed in polynomial time and space, given that they divide the plane into finite number of regions, in the sense as described in \cite{sim}.

\section{Main Construction}
In this section, we construct a closed plane curve $\gamma$ such that the closed curve $2\cdot \gamma$ formed by traversing $\gamma$ twice is much easier to contract to a point than $\gamma$, in the sense of area of the homotopy. By the geometric interpretation of the weighted cancellation norm, it suffices to construct an element $\overline{w}$ such that the ratio of $||\overline{w}^2||$ to $||\overline{w}||$ is small. 

As a weak example, by taking $\overline{w}=abc(cba)^{-1}$, we get $||\overline{w}^2||=6$ and $||\overline{w}||=4$. Then the ratio of $||\overline{w}^2||$ to $||\overline{w}||$ is $\frac{3}2$, which is strictly less than two. The corresponding curve $\gamma$ is illustrated in the figure below.

\begin{figure}[h]
\centering
\includegraphics[width=0.3\textwidth]{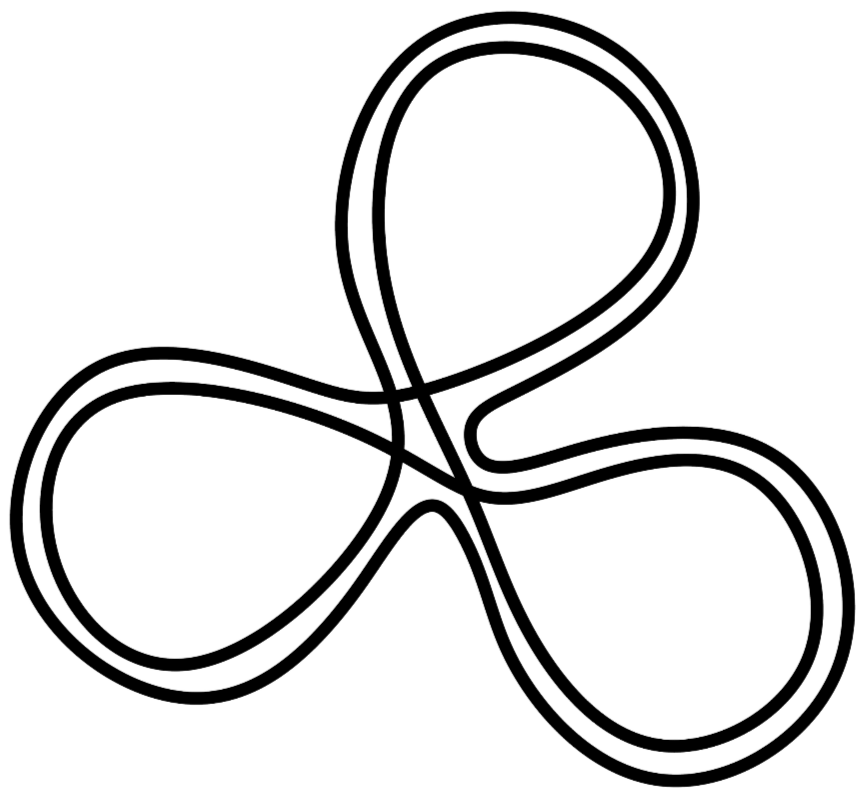}
\caption{\label{fig:curve}A curve $\gamma$ with the ratio of the minimum homotopy area of a null homotopy of $2\cdot \gamma$ to that of $\gamma$ strictly less than two.}
\end{figure}

Let $m\ge 2$ be an integer. Consider the free group $F$ with the standard generator set $T=\{x_i,y_i|i\in \mathbb{Z}^+\cup \{0\}\}$. Let the standard symmetric generating set $S$ be $T\cup T^{-1}$. Let $||\cdot||$ be the cancellation norm on $F$. For every integer $k$, let $F_k$ be the subgroup of $F$ generated by $x_i$, $y_i$ ($0\le i\le k$). For each nonnegative integer $i$, define the twisted element $T_i$ to be $(x_iy_i)^m(y_ix_i)^{-m}\in F$. 

Define the group homomorphisms $f_j: F\rightarrow F$ by $$f_j(a_i)=\begin{cases}T_j^{-1}a_i&\mbox{, if both $i$ and $j$ are odd }\\a_iT_j &\mbox{, otherwise}\end{cases}$$ for each nonnegative integer $j$ and $a_i\in \{x_i,y_i\}$. 

For each integer $j$ with $2\le j\le 2m$, let $f^{(j)}$ denote the group homomorphism $f_{2m}\circ\cdots\circ f_{j+1}$. 

Our construction for $\overline{w}\in F$ is given by the formula $\overline{w}=f^{(2)}(x_0)$. For each element $w\in F$ represented by a word of length $l$ and a nonnegative integer $j$, there exists a word of length $(4m+1)l$ representing $f_j(w)$, so $\overline{w}$ is represented by a word of length $(4m+1)^{2m-2}$. We prove later in this section that $||\overline{w}||=\Theta(m^2)$ and $||\overline{w}^2||=\Theta(m)$ as $m$ tends to infinity. Therefore, by taking a sufficiently large integer $m$, the element $f^{(2)}(x_0)$ serves as an example of an element $\overline{w}\in F$ with a small value of $\frac{||\overline{w}^2||}{||\overline{w}||}$.

More specifically, we establish the following bounds in subsequent subsections.
\begin{theorem}\label{main} For each integer $m\ge 2$, the group element $\overline{w}= f^{(2)}(x_0)$ meets the following conditions:
\begin{enumerate}[(a)]
\item $\overline{w}$ is represented by a word of length $(4m+1)^{2m-2}$;
\item $m^2-3m+5 \le ||\overline{w}||\le 4m^2-8m+5 $;
\item $2 m - 2\le ||\overline{w}^2||\le 8m-4$.
\end{enumerate}
\end{theorem}

Theorem $\ref{main}$ allows us to construct the desired closed plane curve in the following way.

\begin{theorem}\label{main-curve}
For any positive real number $\varepsilon$, there exists a piecewise smooth closed curve $\gamma$ on the plane with $\mathrm{Area}(H_2)<\varepsilon\mathrm{Area}(H_1)$, where $H_1$ (resp., $H_2$) is a piecewise smooth homotopy from $\gamma$ (resp., $2\cdot \gamma$) to a constant curve with minimum area.
\end{theorem}
\begin{proof}
Take a positive integer $m$ such that $\frac{8m-4}{m^2-3m+5}<\varepsilon$. Let $R_i$ denote the open square $\{ (x,y)|i<x<i+1,0<y<1\}$ on the plane for each integer $1\le i\le 4m+2$. Let $T$ be the union of boundaries of $R_i$ where $i$ ranges over $1,\cdots,4m+2$, then the fundamental group of $T$ is isomorphic to $F_{2m}$. Let $\gamma$ be a piecewise linear closed plane curve which takes value in $T$ such that $\gamma$ represents $\overline{w}=f^{(2)}(x_0)$ in $\pi_1(T)$, then $2\cdot \gamma$ represents $\overline{w}^2$ in $\pi_1(T)$. By Theorem \ref{main-relation}, the minimum area of a piecewise smooth homotopy from $\gamma$ (resp., $2\cdot \gamma$) to a constant curve is equal to $||\overline{w}||$ (resp., $||\overline{w}^2||$). By Theorem \ref{main}, we know $0<\mathrm{Area}(H_2)<\varepsilon\mathrm{Area}(H_1)$.
\end{proof}
Similar questions involving the closed curve $p \cdot \gamma$ formed by traversing the original curve $p$ times remain open for odd primes $p$. 

\subsection{Upper Bounds for \texorpdfstring{$\overline{w}$ }{Lg} and \texorpdfstring{$\overline{w}^2$ }{Lg}}
Define the group homomorphism $h: F\rightarrow F$ by
$$h(a_i)=\begin{cases} x_0&\mbox{, if }i\mbox{ is even}\\x_1 &\mbox{, otherwise}\end{cases}$$
for $a_i\in\{x_i,y_i\}$. Let $W$ denote the subgroup of $F$ generated by $x_0$, $x_1$, then we have $h(w)\in W$ for any $w\in F$.

The following theorem provides an upper bound for the cancellation norm of $||w^2||$.

\begin{theorem}\label{upper-bound}
Suppose $k$ is a nonnegative integer with $k\le m-1$ and $s_0, s_1$ are two elements in $F_{2m-2k}\cap S$ such that $h(s_0)=x_0$ and that $h(s_1)=x_1$ or $h(s_1)=x_1^{-1}$. Then 
$$||f^{(2m-2k)}(s_0s_1)||\le 4k+2.$$
\end{theorem}
\begin{proof}
We prove the theorem by induction on $k$. The case where $k=0$ is trivial. Now assume that $k=k'$ ($1\le k'\le m-1$) and the theorem is proved for the case where $k=k'-1$. For conciseness, we denote $f^{(2m-2k+2)}$ by $f$ and denote $f^{(2m-2k+1)}=f\circ f_{2m-2k+2}$ by $\tilde{f}$. 

If $h(s_1)=x_1$, then for $j=2m-2k+2$, we have \begin{align*}&||f^{(2m-2k)}(s_0s_1)||\\=&||f\circ f_{j}\circ f_{j-1}(s_0s_1)||\\=&||f\circ f_{j}(s_0s_1)|| \\=&||f(s_0 T_{j} s_1 T_{j})||\\=&||f(s_0 (x_j y_j)^m (y_j x_j)^{-m} s_1 x_j (y_j x_j)^{m}x_j^{-1} y_j (x_j y_j)^{-m}y_j^{-1})||\\=&||f((x_jy_j)^{-m}y_j^{-1}s_0(x_j y_j)^{m})f((y_j x_j)^{-m}s_1 x_j (y_j x_j)^m)f(x_j^{-1}y_j)||\\\le&||f(y_j^{-1}s_0)||+||f(s_1 x_j)||+||f(x_j^{-1} y_j)||\\=& ||f(s_1 x_j)||+ ||y_j^{-1}s_0||+||x_j^{-1} y_j||\\=&||f(x_j s_1)||+4\\\le &4k+2\end{align*}
by inductive hypothesis.

If $h(s_1)=x_1^{-1}$, then for $j=2m-2k+1$, we have \begin{align*}&||f^{(2m-2k)}(s_0s_1)||\\=&||\tilde{f}\circ f_{j}(s_0s_1)||\\=&||\tilde{f}(s_0 T_j s_1 T_j)|| \\=&||\tilde{f}(s_0 (x_j y_j)^m (y_j x_j)^{-m} s_1 x_j (y_j x_j)^{m}x_j^{-1} y_j (x_jy_j)^{-m}y_j^{-1})||\\=&||\tilde{f}((x_jy_j)^{-m}y_j^{-1}s_0(x_j y_j)^{m})\tilde{f}((y_j x_j)^{-m}s_1 x_j (y_j x_j)^m)\tilde{f}(x_j^{-1}y_j)||\\\le&||\tilde{f}(y_j^{-1}s_0)||+||\tilde{f}(s_1 x_j)||+||\tilde{f}(x_j^{-1} y_j)||\\=&||f\circ f_{j+1}(s_0 y_j^{-1})||+ ||s_1 x_j||+ ||x_j^{-1} y_j||\\=&||f(s_0 y_j^{-1})||+4\\\le &4k+2\end{align*}
by inductive hypothesis.

Therefore the theorem for $k=k'$ is proved and we are done.
\end{proof}
\begin{corollary}\label{upper-bound-col}
Suppose $k$ is a nonnegative integer with $k\le m-1$ and $w$ is an element in $W$ represented by a word of even length $n$. Then 
$$||f^{(2m-2k)}(w)||\le (4k+2)n.$$
\end{corollary}
\begin{proof}
Since each element in $W$ represented by a word of length $2$ can be written as a product of two elements which are conjugates of $x_0x_1$ or $x_0x_1^{-1}$ or their inverses, by Theorem \ref{upper-bound}, the corollary holds for $n=2$. Because any element in $W$ represented by a word of even length $n$ can be written as a product of $\frac{n}{2}$ elements in $W$ represented by a word of length $2$, it holds for any even number $n\ge 0$.
\end{proof}

By Corollary $\ref{upper-bound-col}$, we have $||\overline{w}^2||=||f^{(2)}(x_0^2)||\le 8m-4$.

The following theorem provides an upper bound for the cancellation norm of $||w||$.

\begin{theorem}\label{new-upper-bound}
Suppose $k$ is a nonnegative integer with $k\le m-1$ and $s$ is an element in $F_{2m-2k}\cap S$ such that $h(s)=x_0$. Then
$$||f^{(2m-2k)}(s)||\le 4k^2+1$$
\end{theorem}
\begin{proof}
Use induction on $k$. The case where $k=0$ is trivial. Now assume that $k=k'$ ($1\le k'\le m-1$) and the theorem is proved for the case where $k=k'-1$. As in the proof of Theorem $\ref{upper-bound}$, we denote $f^{(2m-2k+2)}$ by $f$ and denote $f^{(2m-2k+1)}=f\circ f_{2m-2k+2}$ by $\tilde{f}$. Then for $j=2m-2k+1$, we have
\begin{align*}
&||f^{(2m-2k)}(s)||\\=&||\tilde{f}\circ f_{j}(s )||
\\=&||\tilde{f}(sT_{j})||
\\=&||\tilde{f}(s (x_{j}y_{j})^m (y_{j}x_{j})^{-m})|| 
\\=&||\tilde{f}(s y_j^{-1} (y_j x_{j})^{m} y_j (y_j x_j)^{-m})|| 
\\=&||\tilde{f}(s y_j^{-1})\tilde{f}((y_j x_{j})^{m} y_j (y_j x_j)^{-m})||
\\\le&||\tilde{f}(y_j)||+||\tilde{f}(s y_j^{-1} )||
\\=&||f(y_j T_{j+1})||+||f(s y_j^{-1})||
\\\le&||f(y_j (x_{j+1}y_{j+1})^m(y_{j+1}x_{j+1})^{-m})||+4k-2
\\=&||f(y_j y_{j+1}^{-1}(y_{j+1}x_{j+1})^m y_{j+1} (y_{j+1}x_{j+1})^{-m} )||+4k-2
\\=&||f(y_j y_{j+1}^{-1})f((y_{j+1}x_{j+1})^m y_{j+1} (y_{j+1}x_{j+1})^{-m} )||+4k-2
\\=&||f(y_{j+1})||+||f(y_j y_{j+1}^{-1})||+4k-2
\\\le&||f(y_{j+1})||+8k-4
\\\le&4(k-1)^2+1+8k-4
\\=& 4k^2+1
\end{align*}
by inductive hypothesis and Theorem $\ref{upper-bound}$.
Therefore the theorem for $k=k'$ is proved and we are done.
\end{proof}
By Theorem $\ref{new-upper-bound}$, we have $||\overline{w}||=||f^{(2)}(x_0)||\le 4m^2-8m+5$.

\subsection{Lower Bounds for \texorpdfstring{$\overline{w}$ }{Lg} and \texorpdfstring{$\overline{w}^2$ }{Lg}}

Before introducing the lower bounds for $||\overline{w}||$ and $||\overline{w}^2||$, we need several technical lemmas for the inductive step in the proof of Theorem \ref{lower-bound}. 

\begin{lemma}\label{uv-lemma}
For each integer $i$ with $2\le i\le2 m$, there exist group elements $U'_i$, $V'_i$, $U''_i$ and $V''_i$ satisfying the following conditions:
\begin{enumerate}[(a)]
\item $U'_i$, $V'_i$, $U''_i$ and $V''_i$ are generated by $\{x_j,y_j | i<j\le 2m\}$;
\item $h(U'_i)=h(V'_i)=h(U''_i)=h(V''_i)=1$;
\item For each integer $j$ with $0\le j\le i$ and $a_j\in\{x_j,y_j\}$, we have $$f^{(i)}(a_j)= \begin{cases}U'_i a_j V'_i&\mbox{, if $j$ is even}\\ U''_i a_j V''_i&\mbox{, if $j$ is odd}. \end{cases}$$
\end{enumerate}
\end{lemma}
\begin{proof}
Use induction on $2m-i$. For the case $i=2m$, we take $U'_i=V'_i=U''_i=V''_i=1$ and the conditions hold. Now assume that $i=i'$ ($2\le i'< 2m$) and the statement is true for $i=i'+1$. Then \begin{align*}
f^{(i)}(a_j)&=f^{(i+1)}\circ f_{i+1} (a_j)\\
&=\begin{cases}f^{(i+1)}(T_{i+1}^{-1}a_j)&\mbox{, if $i$ is even and $j$ is odd}\\f^{(i+1)}(a_jT_{i+1})&\mbox{, otherwise}\end{cases}\\
&=\begin{cases}f^{(i+1)}(T_{i+1})^{-1} U''_{i+1}a_j V''_{i+1}&\mbox{, if $i$ is even and $j$ is odd} \\U''_{i+1} a_j V''_{i+1} f^{(i+1)}(T_{i+1}) &\mbox{, if both $i$ and $j$ are odd}\
\\U'_{i+1} a_j V'_{i+1} f^{(i+1)}(T_{i+1}) &\mbox{, if $j$ is even}\end{cases}
\end{align*}
for each integer $j$ with $0\le j\le i$ and $a_j\in \{x_j,y_j\}$. By taking \begin{align*}U'_i&=U'_{i+1}\\V'_i&=V'_{i+1}f^{(i+1)}(T_{i+1})\\U''_i&=\begin{cases}f^{(i+1)}(T_{i+1})^{-1}U''_{i+1}&\mbox{, if $i$ is even}\\U''_{i+1}&\mbox{, if $i$ is odd}\end{cases}\\V''_i&=\begin{cases}V''_{i+1}&\mbox{, if $i$ is even}\\V''_{i+1}f^{(i+1)}(T_{i+1})&\mbox{, if $i$ is odd}\end{cases}\end{align*}
the condition (c) holds. Because $f^{(i+1)}(T_{i+1})$ is generated by $\{x_j, y_j|i<j\le 2 m\}$ and $h\circ f^{(i+1)}(T_{i+1})=h(T_{i+1})=1$, the conditions (a) and (b) also hold. Therefore the statement for $i=i'$ is proved and we are done.
\end{proof}

\begin{lemma}\label{cancel-out-lemma}
Suppose that $k\ge 2$ is an integer and $w$ is an element in $W$. Let $f:F_{k}\rightarrow F$ be a group homomorphism such that $f(a_i)=U_i a_i V_i$ where $U_i, V_i$ are generated by $\{x_j,y_j|k<j \}$ with $h(U_i)=h(V_i)=1$ for each $0\le i\le k$ and $a_i\in \{x_i,y_i\}$. Then $f(f_k(w))$ and $f(w)$ are conjugate elements if and only if at least one of the following statements holds:
\begin{enumerate}[(a)]
\item There exists an integer $t$, such that $w$ is a conjugate of $(x_0x_1^{-1})^t$ and $k$ is even;
\item There exists an integer $t$, such that $w$ is a conjugate of $(x_0x_1)^t$ and $k$ is odd.
\end{enumerate}

\end{lemma}
\begin{proof}
Suppose that the conjugacy class of $w$ is represented by a cyclically reduced word $s_1 \cdots s_n$ where the letters $s_1,\cdots,s_n\in \{x_0,x_0^{-1},x_1,x_1^{-1}\}$. Then $f(f_{k}(w))$ is a conjugate of $f(f_{k}(s_1))\cdots f(f_{k}(s_n))$. By definition of $f_{k}$, we know $f(f_{k}(s_i))$ is either $f(s_i) f(T_{k})$ or $f(T_{k})^{-1} f(s_i)$ for $i=1,\cdots, n$. On one hand, if $f(f_k(w))$ and $f(w)$ are conjugate elements, then there is no $x_{k}$, $x_{k}^{-1}$, $y_{k}$ or $y_{k}^{-1}$ in a cyclically reduced word representing the conjugacy class of $f(f_{k}(w))$. In this case, all $f(T_{k})$'s and $f(T_{k})^{-1}$'s cancel out cyclically, which implies that all $T_{k}$'s and $T_{k}^{-1}$'s cancel out in a cyclically reduced word representing the conjugacy class of $f_{k}(w)$, so the condition (a) or (b) holds. On the other hand, if the condition (a) or (b) holds, then we can verify that all $f(T_{k})$'s and $f(T_{k})^{-1}$'s cancel out cyclically in a cyclically reduced word representing the conjugacy class of $f(f_{k}(w))$, so $f(f_{k}(w))$ and $f(w)$ are conjugate elements.
\end{proof}

\begin{lemma}\label{many-unpaired-position-lemma}
Suppose that $k\ge 2$ is an integer and $w$ is an element in $W$. Let $f:F_{k}\rightarrow F$ be a group homomorphism such that $f(a_i)=U_i a_i V_i$ where $U_i, V_i$ are generated by $\{x_j,y_j|k<j \}$ with $h(U_i)=h(V_i)=1$ for each $0\le i\le k$ and $a_i\in \{x_i,y_i\}$. Let $\mathcal{F}$ be a folding of a cyclically reduced word in $S$ representing the conjugacy class of $f(f_{k}(w))$ with the number of unpaired positions equal to $||f(f_{k}(w))||$, and let $t$ be the number of unpaired positions in $\mathcal{F}$ at which the letters are in $\{x_k,x_k^{-1},y_k,y_k^{-1}\}$. Then we have $||f(f_{k}(w))||\ge ||f(w)||+t$.
\end{lemma}
\begin{proof}
Define the group homomorphism $g_k:F\rightarrow F$ by $$g_k(a_i)=\begin{cases}a_i &\mbox{, if }i\neq k\\1 &\mbox{, otherwise}\end{cases}$$ for each $a_i\in\{x_i,y_i\}$. Because \begin{align*} g_k(f(T_{k}))&=g_k(f((x_{k}y_{k})^m(y_{k}x_{k})^{-m}))\\
&=(g_k(f(x_{k}))g_k(f(y_{k})))^m(g_k(f(y_{k}))g_k(f(x_{k})))^{-m}\\
&=(U_{k}V_{k})^{2m}(U_{k}V_{k})^{-2m}\\&=1,\end{align*}
we obtain $g_k(f( f_{k}(w)))=f(w)$. Since the folding $\mathcal{F}$ induces a folding $\mathcal{F}'$ of the word $g_k(f(f_{k}(w)))$ such that the number of unpaired positions in $\mathcal{F}$ is at least $t$ more than the number of unpaired positions in $\mathcal{F}'$, we have $||f(f_{k}(w))||\ge||g_k(f(T_{k}))||+t= ||f(w)||+t$.
\end{proof}
\begin{lemma}\label{single-letter-lemma}
Suppose that $k\ge 2$ is an integer and $w$ is a conjugate of $x_0$, $x_0^{-1}$, $x_1$ or $x_1^{-1}$. Let $f:F_{k}\rightarrow F$ be a group homomorphism such that $f(a_i)=U_i a_i V_i$ where $U_i, V_i$ are generated by $\{x_j,y_j|k<j \}$ with $h(U_i)=h(V_i)=1$ for each $0\le i\le k$ and $a_i\in \{x_i,y_i\}$. Then at least one of the following statements holds:
\begin{enumerate}[(a)]
\item $||f(f_{k}(w))||\ge ||f(w)||+4m$;
\item $w$ is a conjugate of $x_0$ or $x_0^{-1}$ and $k$ is even;
\item $w$ is a conjugate of $x_1$ or $x_1^{-1}$ and $k$ is odd;

\item There exist two nontrivial elements $u,v\in W$ such that the elements $uv$ and $w$ are conjugate and $||f(f_{k}(w))||\ge||f(u)||+||f(v)||$.
\end{enumerate}
\end{lemma}
\begin{proof}
Without loss of generality, we may assume $w\in\{x_0, x_0^{-1},x_1,x_1^{-1}\}$, then $f(f_k(w))$ equals $f(T_k)^{-1}f(w)$ or $f(w)f(T_k)$, and there are $4m$ positions at which the letters are in $\{x_k,x_k^{-1},y_k,y_k^{-1}\}$ in a cyclically reduced word in $S$ representing the conjugacy class of $f(f_{k}(w))$. Let $\mathcal{F}$ be a folding of a cyclically reduced word in $S$ representing the conjugacy class of $f(f_{k}(w))$ with the number of unpaired positions equal to $||f(f_{k}(w))||$. By Lemma $\ref{many-unpaired-position-lemma}$, there are at least one letter in $\{x_k,x_k^{-1},y_k,y_k^{-1}\}$ not at the unpaired positions in $\mathcal{F}$, or the statement (a) is true.

Now assume the statement (a) does not hold.

Let $a$ be a letter in $\{x_k,x_k^{-1},y_k,y_k^{-1}\}$ at position $j_1$ paired with $a^{-1}$ at position $j_2$ in the folding $\mathcal{F}$. By taking a conjugation, we can rewrite the cyclically reduced word representing the conjugacy class of $f_{k}(w)$ as $a u_0 a^{-1}v_0$ where the letters $a$ and $a^{-1}$ correspond to the letters at position $j_1$ and $j_2$ in the word $f(a u_0 a^{-1}v_0)=U_{k}a V_{k}\cdot f(u_0)\cdot V_{k}^{-1} a^{-1} U_{k}^{-1}\cdot f(v_0)$ representing the conjugacy class of $f(f_{k}(w))$ and $u_0,v_0\in F_{k}$ without loss of generality. Moreover, we can assume that $u_0$ is generated by elements in $\{x_k,x_k^{-1},y_k,y_k^{-1}\}$ and represented by a word of odd length. Then, $h(v_0)=1$ implies that $w$ is a conjugate of $x_0$ or $x_0^{-1}$ if $k$ is even, and that $w$ is a conjugate of $x_1$ or $x_1^{-1}$ if $k$ is odd. So the statement (b) or (c) holds or we have $h(v_0)\neq 1$. 

Define $u$ (resp., $v$) to be $h(a u_0 a^{-1})\in W$ (resp., $h(v_0)\in W$), then $u$ is also represented by a word of odd length, $v\neq 1$ and $u v$ is a conjugate of $h(f_k(w))=w$. The folding $\mathcal{F}$ induces a folding $\mathcal{F}_1$ of the word $V_{k}f(u_0)V_{k}^{-1}$ and a folding $\mathcal{F}_2$ of the word $U_{k}U_{k}^{-1} f(v_0)$ by pairing the corresponding pairs in $\mathcal{F}$. Since the sum of the numbers of unpaired positions in $\mathcal{F}_1$ and $\mathcal{F}_2$ is equal to the number of unpaired positions in $\mathcal{F}$, the cancellation norm $||f(f_{k}(w))||$ is at least $||f(u_0)||+||f(v_0)||$. Define the group homomorphisms $\tilde{h}_j:F\rightarrow F$ by $$\tilde{h}_j(a_i)=\begin{cases}h(a_i) &\mbox{, if }i\le j\\a_i &\mbox{, otherwise}\end{cases}$$ for $a_i\in\{x_i,y_i\}$. Then we have \begin{align*}||f(f_k(w))||&\ge ||f(u_0)||+||f(v_0)||\\&\ge ||\tilde{h}(f(u_0))||+||\tilde{h}(f(v_0))||\\&=||f(h(u_0))||+||f(h(v_0))||\\&=||f(u)||+||f(v)||.\end{align*} So at least one of the statements (b), (c) and (d) is true. 
\end{proof}
\begin{lemma}\label{multi-letter-lemma}
Suppose that $k\ge 2$ is an integer and $w$ is an element in $W$ whose cyclic length is at least two. Let $f:F_{k}\rightarrow F$ be a group homomorphism such that $f(a_i)=U_i a_i V_i$ where $U_i, V_i$ are generated by $\{x_j,y_j|k<j \}$ with $h(U_i)=h(V_i)=1$ for each $0\le i\le k$ and $a_i\in \{x_i,y_i\}$. Then at least one of the following statements holds:
\begin{enumerate}[(a)]
\item $f(f_k(w))$ and $f(w)$ are conjugate;
\item There does not exist two consecutive $x_0$'s or $x_0^{-1}$'s in any cyclic reduced word representing the conjugacy class of $w$, and there exists two consecutive $x_1$'s or $x_1^{-1}$'s in a cyclic reduced word representing the conjugacy class of $w$ and $k$ is even;
\item There does not exist two consecutive $x_1$'s or $x_1^{-1}$'s in any cyclic reduced word representing the conjugacy class of $w$, and there exists two consecutive $x_0$'s or $x_0^{-1}$'s in a cyclic reduced word representing the conjugacy class of $w$ and $k$ is odd;
\item $||f(f_{k}(w))||\ge ||f(w)||+4m$;
\item There exist two nontrivial elements $u,v\in W$ such that the elements $uv$ and $w$ are conjugate and $||f(f_{k}(w))||\ge||f(u)||+||f(v)||$;
\item $||f(f_{k}(w))||\ge ||f(w)||+2$ and $w$ is represented by a word of even length.
\end{enumerate}
\end{lemma}
\begin{proof}
Suppose that the conjugacy class of $w$ is represented by a cyclically reduced word $s_1\cdots s_n$ where the letters $s_1,\cdots,s_n\in \{x_0,x_0^{-1},x_1,x_1^{-1}\}$. If the statement (a) does not hold, then there exists at least one $0\le i\le n-1$ such that the $f(T_{k})$ or $f(T_{k})^{-1}$ lying between $f(s_i)$ and $f(s_{i+1})$ does not cancel out cyclically in $f(f_{k}(w))$, where we define $s_0$ to be $s_n$. If neither (b) nor (c) holds, we can either further assume that $s_i=s_{i+1}=x_0$ or $s_i=s_{i+1}=x_0^{-1}$ when $k$ is odd and $s_i=s_{i+1}=x_1$ or $s_i=s_{i+1}=x_1^{-1}$ when $k$ is even, or further assume that $s_i$ is not equal to $s_{i+1}$ and $w$ is represented by a word of even length. By taking a conjugation, we may assume $i=0$ here without loss of generality. Let $\mathcal{F}$ be a folding of a cyclically reduced word in $S$ representing the conjugacy class of $f(f_{k}(w))$ with the number of unpaired positions equal to $||f(f_{k}(w))||$. By Lemma $\ref{many-unpaired-position-lemma}$, not all letters in $\{x_k,x_k^{-1},y_k, y_k^{-1}\}$ in the $f(T_{k})$ or $f(T_{k})^{-1}$ between $s_0$ and $s_1$ are at unpaired positions in $\mathcal{F}$, or the statement (d) is true.

Now assume none of the statements (a), (b), (c) and (d) is true.

Let $a$ be a letter in $\{x_k,x_k^{-1},y_k,y_k^{-1}\}$ at position $j_1$ paired with $a^{-1}$ at position $j_2$ in the folding $\mathcal{F}$ such that the position $j_1$ in $f(f_k(w))$ is either before $f(s_1)$ or after $f(s_n)$. By taking another conjugation, we can rewrite the cyclically reduced word representing the conjugacy class of $f_k(w)$ as $a u_0 a^{-1}v_0$ where the letters $a$ and $a^{-1}$ correspond to the letters at position $j_1$ and $j_2$ in the word $f(a u_0 a^{-1}v_0)=U_{k}a V_{k}\cdot f(u_0)\cdot V_{k}^{-1} a^{-1} U_{k}^{-1}\cdot f(v_0)$ representing the conjugacy class of $f(f_{k}(w))$ and $u_0,v_0\in F_{k}$ without loss of generality.

Define $u$ (resp., $v$) to be the element $h(a u_0 a^{-1})\in W$ (resp., $h(v_0)\in W$), then $uv$ is a conjugate of $h(f_k(w))=w$. Define the group homomorphism $g_k:F\rightarrow F$ by $$g_k(a_i)=\begin{cases}a_i &\mbox{, if }i\neq k\\1 &\mbox{, otherwise}\end{cases}$$ for each $a_i\in\{x_i,y_i\}$. Then there exists an integer $s$ such that $v$ is a conjugate of $g_k(v_0)x_0^{s}$ and $u$ is a conjugate of $g_k(u_0)x_0^{-s}$ if $k$ is even, and that $v$ is a conjugate of $g_k(v_0)x_1^{s}$ and $u$ is a conjugate of $g_k(u_0)x_1^{-s}$ if $k$ is odd. As in the proof of Lemma $\ref{single-letter-lemma}$, we can show that $||f(f_k(w))||\ge ||f(u_0)||+||f(v_0)||$, $||f(u_0)||\ge ||f(u)||$ and $||f(v_0)||\ge ||f(v)||$.

If the positions $j_1$ and $j_2$ lie in the same $f(T_k)$ or $f(T_k)^{-1}$ in the word $f(f_{k}(w))$, then we can further assume that $u_0$ is generated by $\{x_k,y_k\}$ and represented by a word of odd length. Since $u=h(a u_0 a^{-1})$ and $h$ preserves the parity of the word representing the element, the element $u$ is represented by a word of odd length. Because the reduced word representing $g_k(v_0)$ begins with $s_1$ and ends with $s_n$, we have $v$ is nontrivial. So the statement (e) holds. Otherwise, assume the positions $j_1$ and $j_2$ lie in different $f(T_k)$'s or $f(T_k)^{-1}$'s in the word $f(f_{k}(w))$, then the reduced word representing $g_k(u_0)$ begins with $s_{1}$ and the reduced word representing $g_k(v_0)$ ends with $s_n$. If $s_1=s_n$, then both $u$ and $v$ are nontrivial, so the statement (e) holds. Now we assume $s_1\neq s_n$, then $w$ is represented by a word of even length. Because $uv$ is a conjugate of $w$, at least one of $u$ and $v$ is nontrivial. If both $u$ and $v$ are nontrivial, then the statement (e) holds. Otherwise one of $u$ and $v$ equals $1$ and the other is a conjugate of $w$. Since $u$ and $v$ are represented by words of even length, so are $u_0$ and $v_0$. The elements $u_0$ and $v_0$ are nontrivial, which implies $||u_0||,||v_0||\ge 2$, thus $||f(f_k(w))||\ge ||f(u_0)||+||v_0||\ge ||f(w)||+2$ and the statement (f) holds. So at least one of the statements (e) and (f) holds.
\end{proof}

Now we establish lower bounds for the cancellation norms of certain types of group elements.

\begin{theorem}\label{lower-bound}
Suppose $k$ is a positive integer with $k\le m-1$ and $w$ is a nontrivial element in $W$. Then, the following statements hold:
\begin{enumerate}[(a)]
\item If $w$ is represented by a word of even length, then $$||f^{(2m-2k)}(w)||\ge 2k;$$
\item If $w$ is represented by a word of odd length, then $$||f^{(2m-2k)}(w)||\ge k^2-k+3.$$
\end{enumerate}
\end{theorem}
\begin{proof}
We prove the theorem by induction on $k$. Consider the case where $k=k'$ ($1\le k'\le m-1$). If $k\ge 2$, we also assume the theorem is true for $k=k'-1$. 

Define the group homomorphism $g_j:F\rightarrow F$ by $$g_j(a_i)=\begin{cases}a_i &\mbox{, if }i\neq j\\1 &\mbox{, otherwise}\end{cases}$$ for each nonnegative integer $j$ and $a_i\in\{x_i,y_i\}$. By the properties of the group homomorphism $g_j$, we have $||f^{(j-1)}(w_0)||\ge ||g_j\circ f^{(j-1)}(w_0)||=||f^{(j)}(w_0)||$ for any $w_0\in F$ and any integer $j$ with $3\le j\le 2m$. 

By Lemma $\ref{uv-lemma}$, there exist group elements $U'_i, V'_i,U''_i,V''_i$ generated by $\{x_j,y_j|i<j\le 2m\}$ with $h(U'_i)=h(V'_i)=h(U''_i)=h(V''_i)=1$ such that $$f^{(i)}(a_j)= \begin{cases}U'_i a_j V'_i&\mbox{, if $j$ is even}\\ U''_i a_j V''_i&\mbox{, if $j$ is odd} \end{cases}$$ for each integer $j$ with $0\le j\le i$ and $a_j\in\{x_j,y_j\}$. For each $0\le j \le i$ and $a_j\in\{x_j,y_j\}$, by taking $$(U_j,V_j)=\begin{cases}(U'_i,V'_i)&\mbox{, if $j$ is even}\\(U''_i,V''_i)&\mbox{, if $j$ is odd} \end{cases}$$ we have $f^{(i)}(a_j)=U_j a_j V_j$. Therefore, we can apply Lemma $\ref{cancel-out-lemma}$,  Lemma $\ref{single-letter-lemma}$ and Lemma $\ref{multi-letter-lemma}$ to the integer $i$ and the group homomorphism $f^{(i)}$ for any integer $i$ with $2\le i\le 2m$. In particular, we can apply these lemmas to the integer $2m-2k+1$ and the group homomorphism $f^{(2m-2k+1)}$, and the integer $2m-2k+2$ and the group homomorphism $f^{(2m-2k+2)}$. If $w$ is a conjugate of $x_0$, $x_0^{-1}$, $x_1$ or $x_1^{-1}$ and the statement (c) in Lemma $\ref{single-letter-lemma}$ holds for $2m-2k+1$ and $f^{(2m-2k+1)}$, then the statement (b) in Lemma $\ref{single-letter-lemma}$ does not hold for $2m-2k+2$ and $f^{(2m-2k+2)}$. So in this case, the statement (a) or (d) in Lemma $\ref{single-letter-lemma}$ holds for $2m-2k+1$ and $f^{(2m-2k+1)}$ or for $2m-2k+2$ and $f^{(2m-2k+2)}$. If the cyclic length of $w$ is at least two, and the statement (a) or (c) in Lemma $\ref{multi-letter-lemma}$ holds for $2m-2k+1$ and $f^{(2m-2k+1)}$, then by Lemma $\ref{cancel-out-lemma}$, the statements (a) and (b) in Lemma $\ref{multi-letter-lemma}$ do not hold for $2m-2k+2$ and $f^{(2m-2k+2)}$. The converse is true for the same reason. So in this case, the statement (d), (e) or (f) in Lemma $\ref{multi-letter-lemma}$ holds for $2m-2k+1$ and $f^{(2m-2k+1)}$ or for $2m-2k+2$ and $f^{(2m-2k+2)}$.


In case $k\ge 2$, by inductive hypothesis, we know $||f^{(2m-2k+2)}(w_0)||\ge 2k-2$ for any nontrivial element $w_0\in W$ represented by a word of even length, and $||f^{(2m-2k+2)}(w_0)||\ge k^2-3k+5$ for any element $w_0\in W$ represented by a word of odd length.

For the rest of the proof, let $\min_{(u,v)}$ denote the minimum over all pairs of nontrivial elements in $W$ such that $uv$ is a conjugate of $w$.

If $w$ is represented by a word of even length, and the statement (e) or (f) in Lemma $\ref{multi-letter-lemma}$ holds for $2m-2k+1$ and $f^{(2m-2k+1)}$, then
\begin{align*} &|| f^{(2m-2k)}(w)||\\=&||f^{(2m-2k+1)}\circ f_{2m-2k+1}(w)||\\\ge&\min\{ ||f^{(2m-2k+1)}(w)||+2, \min_{(u,v)}\{||f^{(2m-2k+1)}(u)||+||f^{(2m-2k+1)}(v)||\} \}\\\ge& \min\{||f^{(2m-2k)}(w)||+2,\min_{(u,v)}\{||f^{(2m-2k)}(u)||+||f^{(2m-2k)}(v)||\}\}\\ \ge& \begin{cases}\min\{3,2\}&\mbox{, if }k=1\\\min\{2k,4k-4,2k^2-6k+10\}&\mbox{, if }k\ge 2\end{cases}\\=&2k. \end{align*} If $w$ is represented by a word of even length, and the statement (e) or (f) in Lemma $\ref{multi-letter-lemma}$ holds for $2m-2k+2$ and $f^{(2m-2k+2)}$, then 
\begin{align*} &||f^{(2m-2k)}(w)||\\\ge&||f^{(2m-2k+2)}\circ f_{2m-2k+2}(w)||\\\ge&||f^{(2m-2k+1)}||\\=&\min\{ ||f^{(2m-2k+2)}(w)||+2, \min_{(u,v)}\{||f^{(2m-2k+2)}(u)||+||f^{(2m-2k+2)}(v)||\} \}\\ \ge& \begin{cases}\min\{3,2\}&\mbox{, if }k=1\\\min\{2k,4k-4,2k^2-6k+10\}&\mbox{, if }k\ge 2\end{cases}\\=&2k.\end{align*} Therefore the part (a) of the theorem is true for $k=k'$.

If $w$ is represented by a word of odd length, and the statement (d) or (e) in Lemma $\ref{multi-letter-lemma}$ holds for $2m-2k+1$ and $f^{(2m-2k+1)}$ in case the cyclic length of $w$ is greater than $1$, or the statement (a) or (d) in Lemma $\ref{multi-letter-lemma}$ holds for $2m-2k+1$ and $f^{(2m-2k+1)}$ in case $w$ is a conjugate of $x_0$, $x_0^{-1}$, $x_1$ or $x_1^{-1}$, then
\begin{align*} &||f^{(2m-2k)}(w)||\\=&||f^{(2m-2k+1)}\circ f_{2m-2k+1}(w)||\\\ge&\min\{ ||f^{(2m-2k+1)}(w)||+4m, \min_{(u,v)}\{||f^{(2m-2k+1)}(u)||+||f^{(2m-2k+1)}(v)||\} \}\\\ge& \min\{||f^{(2m-2k+2)}(w)||+4m,\min_{(u,v)}\{||f^{(2m-2k+2)}(u)||+||f^{(2m-2k+2)}(v)||\}\}\\ \ge& \begin{cases}\min\{1+4m,3\}&\mbox{, if }k=1\\\min\{k^2-3k+5+4m,k^2-k+3\}&\mbox{, if }k\ge 2\end{cases}\\=&k^2-k+3. \end{align*} If $w$ is represented by a word of odd length, and the statement (d) or (e) in Lemma $\ref{multi-letter-lemma}$ holds for $2m-2k+2$ and $f^{(2m-2k+2)}$ in case the cyclic length of $w$ is greater than $1$, or the statement (a) or (d) in Lemma $\ref{multi-letter-lemma}$ holds for $2m-2k+2$ and $f^{(2m-2k+2)}$ in case $w$ is a conjugate of $x_0$, $x_0^{-1}$, $x_1$ or $x_1^{-1}$, then
\begin{align*} &||f^{(2m-2k)}(w)||\\\ge&||f^{(2m-2k+1)}(w)||\\=&||f^{(2m-2k+2)}\circ f_{2m-2k+2}(w)||\\\ge&\min\{ ||f^{(2m-2k+2)}(w)||+4m, \min_{(u,v)}\{||f^{(2m-2k+2)}(u)||+||f^{(2m-2k+2)}(v)||\} \}\\ \ge& \begin{cases}\min\{1+4m,3\}&\mbox{, if }k=1\\\min\{k^2-3k+5+4m,k^2-k+3\}&\mbox{, if }k\ge 2\end{cases}\\=&k^2-k+3. \end{align*} Therefore the part (b) of the theorem is true for $k=k'$ and we are done.

\end{proof}
By Theorem $\ref{lower-bound}$, we have $||\overline{w}||= ||f_{2m}\circ \cdots \circ f_3(x_0)||\ge m^2-3m+5$ and $||\overline{w}^2||=||f_{2m}\circ \cdots \circ f_3(x_0^2)||\ge 2m-2$.

\section{Reversed Inequality}

In this section, we prove that there exists a positive constant $c_1$ such that
for sufficiently large integer $m$, if a group element $w$ meets the conditions of Theorem \ref{main}, then it cannot be represented by a word of length at most $c_1^m$. Thus, the minimum length of the word representing such an element $w$ must be between 
$c_1^m$ and $m^{c_2m}$ for positive constants $c_1$ and $c_2$ and sufficiently large integer $m$, though the tighter bounds are currently unknown.

\begin{theorem}\label{reversed}
Given a set $T$ and a function $\mathrm{wt}:T\rightarrow \mathbb{R}^+\cup \{0\}$. Let $F_T$ be the group freely generated by $T$. For a group element $w\in F_T$ represented by a word of length at most $2^m$, we have
$$||w||\le \frac{m+1}{2} ||w^2||$$
where $||\cdot||$ denotes the weighted cancellation norm.
\end{theorem}
\begin{proof}
We prove the theorem by induction on $m$. The case where $m=0$ is trivial. Now consider the case where $m=m'>0$ assuming the theorem holds for $m=m'-1$.

Let $l_1\cdots l_n$ be a word representing $w$ over the alphabet $S=T\cup T^{-1}$ with $1\le n\le 2^m$, then the word $l_1\cdots l_n\cdot l_1\cdots l_n$ represents $w^2$. Let $\mathcal{F}$ be a folding of $l_1\cdots l_n\cdot l_1\cdots l_n$ with the sum of weights of the letters at unpaired positions equal to $||w^2||$. For each integer $1\le i\le 2n$ which is not an unpaired position in $\mathcal{F}$, define $f(i)$ as the unique integer such that $\{i,f(i)\}\in \mathcal{F}$. 

Suppose $t'$ is an integer with $1\le t'\le n$. Let $A_{t'}$ denote the set of $i$ such that $t'\le i< n+t'$ is not an unpaired position in $\mathcal{F}$ and that $f(i)\le t$ or $f(i)\ge n+t'$. For each $k$ with $1\le k\le n$, let $P_k$ be the set of pairs $\{i,j\}\in\mathcal{F}$ with $\min\{|i-j|, 2n-|i-j| \}\ge k$. Then because the pairs in $\mathcal{F}$ are pairwise not linked, there are at most $n-k+1$ elements in $P_k$. Therefore, we have $\sum_{t'=1}^n|A_{t'}|=\sum_{t'=1}^n \sum_{i\in A_{t'}}1=\sum_{i=1}^{2n}\sum_{i\in A_{t'}}1 =\sum_{\{i,j\}\in\mathcal{F}}\min\{|i-j|, 2n-|i-j| \} =\sum_{k=1}^n |P_k|\le\sum_{k=1}^n(n-k+1)=\frac{n(n+1)}{2}$. So there exists $t$ with $1\le t\le n$ such that $|A_t|\le 2^{m-1}$.

Let $\psi_t: \{1,\cdots |A_t|\}\rightarrow A_t$ denote the unique increasing bijection. Define $w_t=l_t \cdots l_n\cdot l_1\cdots l_{t-1}$ and $v_t=l_{\psi_t(1)}\cdots l_{\psi_t(|A_t|)}$. On one hand, the set $\mathcal{F}_t^+=\{\{i-t+1,j-t+1\}|\{i,j\}\in\mathcal{F}, t\le i\le n+t-1, t\le j\le n+t-1\}\cup\{\{\psi_t(i)-t+1,n+|A_t|+1-i\}|1\le i\le |A_t|\}$ is a folding of the word $l_t \cdots l_n\cdot l_1\cdots l_{t-1}\cdot l_{\psi_t(|A_t|)}^{-1}\cdots l_{\psi_t(1)}^{-1}$ representing $w_t v_t^{-1}$. The sum of weights of the letters at unpaired positions in $\mathcal{F}_t^+$ is equal to the sum of weights of the letters at unpaired positions $i$ with $t\le i< n+t$ in $\mathcal{F}$. On the other hand, let $\phi_t:\{1,\cdots, t-1\}\cup\{n+t,\cdots, 2n\}\rightarrow \{|A_t|+1,\cdots,|A_t|+n\}$ be an auxiliary function defined by $$\phi_t(i)=\begin{cases}|A_t|+n-t+1+i&\mbox{, if }1\le i<t\\|A_t|-n-t+1+i&\mbox{, if }n+t\le i\le 2n.\end{cases}$$Then the set $\mathcal{F}_t^-=\{\{\phi_t(i),\phi_t(j)\}|\{i,j\}\in \mathcal{F}, i,j\in \{1,\cdots, t-1\}\cup\{n+t,\cdots, 2n\} \}\cup\{\{i,\phi_t\circ f\circ \psi_t(i)\}|1\le i\le |A_t|\}$ is a folding of the word $l_{\psi_t(1)}\cdots l_{\psi_t(|A_t|)}\cdot l_t \cdots l_n\cdot l_1\cdots l_{t-1}$ representing $v_tw_t$. The sum of weights of the letters at unpaired positions in $\mathcal{F}_t^-$ is equal to the sum of weights of the letters at unpaired positions $i$ with $i<t$ or $i\ge n+t$ in $\mathcal{F}$. Therefore the sum of weights of the letters at unpaired positions in $\mathcal{F}_t^+$ and $\mathcal{F}_t^-$ is equal to $||w^2||$, which implies $||w^2||=||w_t^2||\le||w_tv_t^{-1}||+||v_t w_t||\le ||w^2||$. Hence $||w_tv_t^{-1}||+||v_t w_t||= ||w^2||$. Because $v_t$ is represented by a word of length $|A_t|\le 2^{m-1}$, by inductive hypothesis, we have $||v_t||\le\frac{m}{2}||v_t^2||$, so \begin{align*}||w||&=||w_t||\\&=\frac{1}{2}(||w_t||-||w_tv_t^{-1}||)+\frac{1}{2}(||w_t||-||v_t w_t||)+\frac{1}{2}||w^2||\\&\le ||v_t||+\frac{1}{2}||w^2||\\&\le \frac{m}{2} ||v_t^2||+\frac{1}{2}||w^2||\\&\le\frac{m}{2}(||v_t w_t^{-1}||+||w_t v_t ||)+||w^2||\\&=\frac{m}{2}||w^2||+||w^2||\\&=  \frac{m+1}{2}||w^2||.\end{align*}
Therefore the theorem is proved for $m=m'$.
\end{proof}

In fact, with a more detailed discussion we can prove that for a word $l_1\cdots l_n$ of length $n\ge 1$ and a folding $\mathcal{F}$ of the word $l_1\cdots l_n\cdot l_1\cdots l_n$, there exists an integer $1\le t \le n$ such that $|A_t|\le \frac{n+2}{3}$, where $A_t$ is defined as in the preceding proof. This allows the extension of Theorem \ref{reversed} to the case where the group element $w$ is represented by a word of length at most $3^m$.

\section*{Acknowledgements}
This research was conducted at (and funded by) the MIT's Undergraduate Research Opportunities Program for summer 2014, under the direction of Larry Guth. I would like to express my very great appreciation to him for introducing me to this problem, for guiding me through the research process, and for offering me valuable revision suggestions.


\begin{thebibliography}{9}
\bibitem{norm}M. Brandenbursky, \'S. R. Gal, J. K\k{e}dra, M. Marcinkowski, \emph{Cancelation norm and the geometry of biinvariant word metrics}, Glasgow Journal of Mathematics, to appear, preprint available at {\tt arXiv:1310.2921 [math.GT]}, 2013.

\bibitem{sim}E. W. Chambers, Y. Wang, \emph{Measuring similarity between curves on $2$-manifolds via homotopy area}, Proceedings of the Twenty-ninth Annual Symposium on Computational Geometry, 2013, pp. 425-434.

\bibitem{rna}S. Gadgil, \emph{Conjugacy invariant pseudo-norms, representability and RNA secondary structures}, Indian Journal of Pure and Applied Mathematics, Volume 42, Issue 4, 2011, pp. 225-237.

\bibitem{fmorgan}F. Morgan, \emph{Area minimizing currents bounded by higher multiples of curves}, Rendiconti del Circolo Matematico di Palermo, Volume 33, Issue 1, February 1984, pp. 37-46.

\bibitem{bwhite}B. White, \emph{The least area bounded by multiples of a curve}, Proceedings of the American Mathematical Society, Volume 90, Number 2, February 1984, pp. 230-232.

\bibitem{lyoung}L. C. Young, \emph{Some extremal questions for simplicial complexes V. The relative area of a Klein bottle}, Rendiconti del Circolo Matematico di Palermo, Volume 12, Issue 3, September – December 1963, pp. 257-274.

\bibitem{ryoung}R. Young, \emph{Filling multiples of embedded cycles and quantitative nonorientability}, preprint available at {\tt arXiv:1312.0966 [math.DG]}, 2013.





\end{thebibliography}
\end{document}